\documentclass[12pt]{article}

\usepackage{a4wide}
\usepackage{amssymb,amsmath,array,amsthm, amsfonts, amscd}
\usepackage{colortbl}
\usepackage{algpseudocode}
\usepackage{algorithm}
\usepackage{tikz}
\usepackage{graphicx}
\usepackage{comment}
\allowdisplaybreaks[1]

\input xy
\xyoption{matrix}\xyoption{arrow}
\def\edge{\ar@{-}}
\def\dedge{\ar@{.}}

\newtheorem{theorem}{Theorem}[section]
\newtheorem{corollary}[theorem]{Corollary}
\newtheorem{lemma}[theorem]{Lemma}
\newtheorem{proposition}[theorem]{Proposition}

\theoremstyle{definition}
\newtheorem{definition}[theorem]{Definition}
\newtheorem{example}[theorem]{Example}
\newtheorem{remark}[theorem]{Remark}
\newtheorem{conjecture}[theorem]{Conjecture} 
\newtheorem{notation}[theorem]{Notation}

\def\k{K}

\def\mz{{\mathbb Z}}

\def\ad{{\rm ad}}

\def\supp{{\rm Supp}}

\def\oq{\mathcal{O}_{q}} 

\def\qhat{\widehat{q}\,}

\def\oqmmn{\oq(M(m,n))}
\def\oqmkn{\oq(M(k,n))}
\def\oqmkp{\oq(M(k,p))}
\def\oqmkk{\oq(M(k,k))}
\def\oqmnn{\oq(M(n,n))}

\def\oqgkdashn{\oq(G(k',n))}
\def\oqgkn{\oq(G(k,n))}
\def\oqgnminuskn{\oq(G(n-k,n))}
\def\oqgnminuskdash{\oq(G(n-k',n))}
\def\oqgtwofour{\oq(G(2,4))}

\title{Derivations and the first Hochschild cohomology\\ group of the quantum grassmannian\footnote{This research was partly supported by EPSRC grant EP/R009279/1.}}
\author{S Launois and T H Lenagan}
\date{}
\begin{document}
\maketitle
\begin{abstract}
We calculate the derivations and the first Hochschild cohomology group of the quantum grassmannian over a field of characteristic zero in the generic case when the deformation parameter is not a root of unity. Using graded techniques and two special homogeneous normal elements of the quantum grassmannian, we reduce the problem to computing derivations of the quantum grassmannian that act trivially on these two normal elements. We then use the dehomogenisation equality which shows that a localisation of the quantum grassmannian is equal to a skew Laurent extension of quantum matrices. This equality is used to connect derivations of the quantum grassmannian with those of quantum matrices. More precisely, again using  graded techniques, we show that derivations of the quantum grassmannian that act trivially on our two normal elements restrict to homogeneous derivations of quantum matrices. The derivations of quantum matrices are known in the square case and technical details needed to deal with the  general case are given in an appendix. This allows us to explicitly describe the first Hochschild cohomology group of the quantum grassmannian.
\end{abstract} 


\vspace{2ex}

\vspace{2ex}


\section{Introduction} \label{section-introduction}

Let $\k$ denote a field of characteristic zero and $q\in\k$ be a nonzero element  which is not a root of unity. The quantum grassmannian $\oqgkn$ is a noncommutative algebra that is a deformation of the homogeneous coordinate ring of the classical grassmannian of $k$-planes in $n$-space. In this paper, we calculate the derivations and the first Hochschild cohomology group  of the quantum grassmannian.  

One of the motivations for this paper is \cite{ms} where a connection between the totally nonnegative grassmannian -- or rather its cell decomposition into
the union of so-called positroid cells -- and Hochschild (co)homology of $\oqgkn$ was established. More precisely, a link between volume form of positroids and Hochschild (co)homology of $\oqgkn$ is established, allowing the contribution of a positroid cell to the scattering amplitude (in the N = 4 - supersymmetric Yang-Mills theory) to be $q$-deformed. The present paper is the first one in a project to compute the Hochschild (co)homology of quantum positroid varieties. More precisely, the quantum grassmannian can be viewed as a ``quantum positroid variety'' (corresponding to the totally positive grassmannian) thanks to \cite{lln}, and we compute the first Hochschild cohomology group of the quantum grassmannian. We will come back to the general case in future work.

Assume $k\leq n$. Recall that $\oqgkn$ is the subalgebra of the quantum matrix algebra $\oqmkn$ that is generated by the $k\times k$ quantum minors of $\oqmkn$. These $k\times k$ quantum minors are the so-called quantum Pl\"ucker coordinates of $\oqgkn$. They are indexed by $k$-subsets of $\{1, \dots , m\}$ and denoted by $[I]$ for each $k$-subset $I$. The algebra $\oqgkn$ is graded with all quantum Pl\"ucker coordinates homogeneous of degree 1. 

Quantum Pl\"ucker coordinates formed on consecutive columns play a special role: they are normal and each one generates a completely prime ideal of  $\oqgkn$. They are referred to as prime quantum Pl\"ucker coordinates. Among the prime quantum Pl\"ucker coordinates, two are special: $[u]=[1,...,k]$ and $[w]=[n-k+1,...,n]$. The reason for this is that $[u]$ (respectively, $[w]$) $q$-commutes with every quantum Pl\"ucker coordinate $[I]$ with a nonnegative (respectively, nonpositive) power of $q$, that is: 
$$[u][I]=q^{\geq 0} [I][u] \mbox{ and } [w][I]=q^{\leq 0} [I][w].$$ 
We exploit this observation and other graded techniques in order to reduce the problem of computing derivations of  $\oqgkn$ to computing derivations $D$ of  $\oqgkn$ that act trivially on both $[u]$ and $[w]$; that is, derivations $D$ such that $D([u])=D([w])=0$.  

We then employ the dehomogenisation equality which shows that a localisation of the quantum grassmannian is equal to a skew Laurent extension of quantum matrices.
This equality is used to connect the set of derivations of the quantum grassmannian acting trivially on $[u]$ and $[w]$  with that of quantum matrices. More precisely, again using graded techniques, we show that derivations of the quantum grassmannian that act trivially on  $[u]$ and $[w]$ restrict to homogeneous derivations of quantum matrices. To conclude we use our knowledge of derivations of quantum matrices: the set of derivations is known in the case of square quantum matrices, see \cite{ll-qmatrix-ders}. Some technical details that are needed to deal with the fact that the non-square case has not yet been covered are given in an appendix. 

This allow us to to show that the first cohomology group of $\oqgkn$ is a $\k$-vector space of dimension $n$, with basis given by the cosets of the derivations $D_1, \dots , D_n$, where $D_i$ is the derivation of $\oqgkn$ defined by 
$$D_i([I])=\left\{ \begin{array}{ll}
[I] & \mbox{ if }i\in I;\\
0 & \mbox{ otherwise}. 
\end{array}
\right.
$$


\section{Basic definitions} \label{section-basic-definitions}

Throughout the paper, $\k$ denotes a field of characteristic zero and $q\in\k$ is a nonzero element  which is not a root of unity.


The algebra of $m\times n$ quantum matrices over $\k$, denoted by $\oqmmn$, is 
the algebra generated over $\k$ by 
$mn$ indeterminates 
$x_{ij}$, with $1 \le i \le m$ and $1 \le j \le n$,  which commute with the elements of 
$\k$ and are subject to the relations:
\[
\begin{array}{ll}  
x_{ij}x_{il}=qx_{il}x_{ij},&\mbox{ for }1\le i \le m,\mbox{ and }1\le j<l\le
n\: ;\\ 
x_{ij}x_{kj}=qx_{kj}x_{ij}, & \mbox{ for }1\le i<k \le m, \mbox{ and }
1\le j \le n \: ; \\ 
x_{ij}x_{kl}=x_{kl}x_{ij}, & \mbox{ for } 1\le k<i \le m,
\mbox{ and } 1\le j<l \le n \: ; \\
x_{ij}x_{kl}-x_{kl}x_{ij}=(q-q^{-1})x_{il}x_{kj}, & \mbox{ for } 1\le i<k \le
m, \mbox{ and } 1\le j<l \le n.
\end{array}
\]
It is well known that $\oqmmn$ is an iterated Ore extension over $\k$ with the $x_{ij}$ added 
in lexicographic order. An immediate consequence is that $\oqmmn$ is a noetherian domain. 

When $m=n$, the {\em quantum determinant} $D_q$ is defined by;
\[
D_q:= \sum\,(-q)^{l(\sigma)}x_{1\sigma(1)}\dots x_{n\sigma(n)},
\]
where the sum is over all permutations $\sigma$ of $\{1,\dots,n\}$. 

The quantum determinant is a central element in the algebra of quantum matrices $\oqmnn$. \\

If $I$ and $J$ are $t$-element subsets of $\{1,\dots,m\}$ and $\{1,\dots,n\}$, respectively, 
then the {\em quantum minor} $[I\mid J]$ is defined to be the quantum determinant of the 
$t\times t$ quantum matrix subalgebra generated by the variables $x_{ij}$ with $i\in I$ and $j\in J$. \\

 The quantum matrix algebra $\oqmmn$  is a connected $\mathbb{N}$-graded algebra with each generator $x_{ij}$ given degree one. Note that each $t \times t$ quantum minor has degree $t$.\\

\begin{definition} 
Let $k\leq n$. 
The {\em homogeneous coordinate ring of the $k\times n$ quantum grassmannian}, $\oqgkn$ (known informally as the {\em quantum grassmannian}) is the subalgebra of the quantum matrix algebra $\oqmkn$ that is generated by the $k\times k$ quantum minors of $\oqmkn$, see, for example, \cite{klr}. 
\end{definition} 

A $k\times k$ quantum minor of $\oqmkn$ must use all of the $k$ rows, and so we can specify the quantum minor by specifying the columns that define it. With this in mind, we will write 
$[J]$ for the quantum minor $[1,\dots,k\mid J]$, for any $k$-element subset $J$ of $\{1,\dots,n\}$. Quantum minors of this type are called {\em quantum Pl\"ucker coordinates}. The quantum grassmannian $\oqgkn$ is a connected $\mathbb{N}$-graded algebra with each quantum Pl\"ucker coordinate given degree one. The set of quantum Pl\"ucker  coordinates in $\oqgkn$ is denoted by $\Pi$. There is a natural partial order on $\Pi$ defined in the following way: if $I=[i_1<\dots<i_k]$ and $J=[j_1<\dots<j_k]$ then $[I]<[J]$ if and only if 
$i_l\leq j_l$ for each $l=1,\dots,k$.  A  {\em standard monomial} in the quantum Pl\"ucker coordinates is an expression of the form $[I_1][I_2]\dots[I_t]$ where $I_1\leq I_2\leq\dots\leq I_t$ in this partial order. The set of all standard monomials forms a vector space basis of $\oqgkn$ over $\k$, see, for example, \cite[Corollary 2.1]{klr}. We will refer to the process of rewriting a monomial in terms of standard monomials as {\em straightening}.\\

The quantum grassmannian $\oq(G(1,n))$ is a quantum affine space, and, as such, its 
derivations are known, see \cite[1.3.3 Corollaire]{ac}; so we will assume throughout this paper that $k>1$. As $\oq(G(n-1,n))\cong \oq(G(1,n))$, by \cite[Proposition 3.4]{ll-aut-qg}, we exclude the case $k=n-1$ as well. Thus, we will assume throughout the paper that $2\leq k\leq n-2$, and so $n\geq4$.\\

Also, it is shown in \cite[Proposition 3.4]{ll-aut-qg} that 
$\oq(G(k,n))\cong \oq(G(n-k,n))$, when $2k\leq n$. In Section~\ref{section-transferring-derivations} we need to assume that $2k\leq n$ and 
the general result is obtained from this case in Subsection~\ref{subsection-general-case} by using the isomorphism just mentioned.



\section{Derivations for $\oqgkn$ inherited from $\oqmkn$}\label{section-derivations-inherited}

We use the following notation for the delta truth function: if $P$ is a proposition then $\delta(P)=1$ if $P$ is true, while $\delta(P)=0$ if $P$ is false. As is traditional, we 
write $\delta_{ij}$ for $\delta(i=j)$\\

Recall from \cite{ll-qmatrix-ders} that for each column of $\oqmkn$ there is a derivation that is the identity on any quantum minor that contains that column, and is zero on the other quantum minors. These derivations restrict to the subalgebra $\oqgkn$ to give us $n$ derivations which we denote by $D_i$ for $i=1,\dots,n$. We refer to these derivations as the {\em column derivations} of $\oqgkn$.

\begin{lemma} \label{lemma-column-derivations} 
For each $i=1,\dots, n$, there is a derivation $D_i$ whose action on quantum Pl\"ucker coordinates is given by $D_{i}([I])=\delta(i\in I)[I]$.
\end{lemma} 

\begin{proof} This is immediate from the definition of $D_i$ as the restriction to $\oqgkn$ of a column derivation of $\oqmkn$.
\end{proof} 

\begin{corollary}\label{corollary-sum=identity}
\[
\frac{1}{k}\left(\sum_{i=1}^n D_i\right)([I])=[I],
\]
for each quantum Pl\"ucker coordinate $[I]$ in $\oqgkn$.
\end{corollary}

\begin{proof}
This follows immediately from the previous lemma, as there are $k$ occurences of $D_i$ for which 
$D_i([I])=[I]$, and otherwise $D_i([I])=0$.
\end{proof}

Our main aim in this paper is to prove the following conjecture. 

\begin{conjecture} \label{conjecture}
Let $\k$ be a field of characteristic zero and let $q$ be a nonzero element of $\k$ that is not a root of unity. Then 
every derivation of $\oqgkn$ can be written as a linear combination  of  inner derivations and column derivations .  Furthermore, the column derivations are linearly independent modulo the space generated by the inner derivations.
\end{conjecture} 

As a consequence of the truth of this conjecture we identify the first Hochschild cohomology group of the quantum grassmannian.



\section{The dehomogenisation equality for $\oqgkn$} \label{section-dehomogenisation-equality}

In this section, we recall results concerning the dehomogenisation equality that occur in \cite[Section 3]{ll-aut-qg}. We refer the reader to that paper for detailed definitions and proofs. \\

Set $u=\{1,\dots, k\}$. Then $[u]$ commutes with all other quantum Pl\"ucker coordinates up to a power of $q$,  by \cite[Lemma 3.1]{ll-aut-qg}. 
 As $\oqgkn$ is generated by the  quantum Pl\"ucker coordinates it follows that the element $[u]$ is a normal element and so we may invert $[u]$ to obtain the overring $T:=\oqgkn[[u]^{-1}]$. \\

For $1\leq i\leq k$ and $1\leq j\leq n-k$, set

\[
x_{ij}:=[1\dots,\widehat{k+1-i},\dots k, j+k][u]^{-1}\in T.
\]
Set $p:=n-k$. The elements $x_{ij}$ generate a subalgebra of $T$ that is a quantum matrix algebra $\oqmkp$. In fact, $T$ is generated over $\oqmkp$ by $[u]^{\pm1}$, so 
\begin{eqnarray*}
T=\oqgkn[[u]^{-1}]=
\oqmkp[[u]^{\pm1}; \sigma],
\end{eqnarray*}
where $\sigma$ is the automorphism of $\oqmkp$ defined by $\sigma(x_{ij}):=qx_{ij}$ (since $[u]x_{ij}=qx_{ij}[u]$ for all $i,j$). 

When we operate on the right hand side of this equality, we will write $y$ for $[u]$. Thus, $yx_{ij}=\sigma(x_{ij})y$ and 

\begin{eqnarray}
\label{equation-dehomogenisation}
T=\oqgkn[[u]^{-1}]=
\oqmkp[y,y^{-1}; \sigma]
\end{eqnarray}

We refer to this equality as the {\em dehomogenisation equality}. We will 
make extensive use of the dehomogenisation equality to transfer derivations from $\oqmkp$ to $\oqgkn$ and vice-versa. \\



Note that \cite[Lemma 3.2]{ll-aut-qg} gives the formula for general quantum 
minors of the algebra $\oqmkp$ when viewed as elements in $\oqgkn[[u]^{-1}]$, and each quantum Pl\"ucker coordinate, multiplied by $[u]^{-1}$, occurs in the 
formula. The formula is
\[
[I|J]= [\{1,\dots, k\}\backslash (k+1-I)\sqcup (k+J)][u]^{-1}\,.
\]
In the reverse direction, for a quantum Pl\"ucker coordinate $[L]$ of $\oqgkn$, we have 
\[
[L]=[L_{\leq k}\sqcup L_{>k}] = [I\mid J][u]\,,
\]
where $L_{\leq k}:=L\cap\{1,\dots,k\}$ and $L_{>k}:= L\cap\{k+1,\dots,n\}$, while 
$I= \{(k+1)-(\{1,\dots,k\}\backslash L_{\leq k}\}$ and $J= L_{>k}-k$.




\section{Derivations arising via  dehomogenisation} \label{section-via-dhom}

The $n$ column derivations $D_i$ of $\oqgkn$ that we have defined in Section~\ref{section-derivations-inherited} satisfy $D_i([u])=[u]$, for $1\leq i\leq k$ and  $D_i([u])=0$, for $i<k\leq n$; so they extend to $n$ derivations $\widetilde{D_i}$ of $T=R[y,y^{-1};\sigma] = \oqgkn[u]^{-1}$ with  
$\widetilde{D_i}(y)=y$, for $1\leq i\leq k$, and $\widetilde{D_i}(y)=0$ , for  $i<k\leq n$.\\

 We will show that  these derivations $\widetilde{D_i}$  of $T$ 
 coincide with extensions of known derivations of $\oqmkp$ which we now recall from \cite{ll-qmatrix-ders}. \\
 
 For $1\leq i\leq k$, there are derivations $D_{i \ast }$ of $\oqmkp$ defined by 
 $D_{i \ast }(x_{rs}):=\delta_{ir}x_{rs}$, and for $1\leq j\leq p$, there are derivations $D_{\ast j}$ of $\oqmkp$ defined by 
 $D_{\ast j}(x_{rs}):=\delta_{js}x_{rs}$. 
 In other words, $D_{i \ast }$ fixes row $i$ and kills all the other rows of $\oqmkp$, while $D_{\ast j}$ fixes column $j$ and kills all other columns of $\oqmkp$. For these observations, see the comment after \cite[Corollary 2.8]{ll-qmatrix-ders}.\\
 
 Note that the Leibniz formula for derivations shows that $D_{i \ast }([I\mid J])=\delta(i\in I)[I\mid J]$ and that $D_{\ast j}([I\mid J])=\delta(j\in J)[I\mid J]$ for any quantum minor $[I\mid J]$ of $\oqmkp$. The derivations that we have just defined are not linearly independent: it is easy to check that 
 $\sum_{i=1}^k D_{i \ast }=\sum_{j=1}^p D_{\ast j}$.\\
 
 We will extend these derivations of $\oqmkp$ to derivations $\widetilde{D_{i \ast }}$ and $\widetilde{D_{\ast  j}}$ of $T =\oqmkp[y^{\pm 1};\sigma]=\oqgkn[u]^{-1}$ as follows. 
For each $i = 1,\dots, k$ set $\widetilde{D_{i \ast }}(y)=y$ and note that by applying the Leibniz formula for derivatives to $1=y y^{-1}$ we obtain $\widetilde{D_{i \ast }}(y^{-1})= -y^{-1}$.
For each $j=1,\dots, p$ set $\widetilde{D_{\ast j}}(y)=\widetilde{D_{\ast j}}(y^{-1})=0$. These choices for the action on $y$ are 
necessary to obtain Corollary~\ref{corollary-derivation-equalities}.\\

In the next proposition, we calculate the effect of the derivations $\widetilde{D_{i\ast }}$ and $\widetilde{D_{\ast j}}$ on quantum Pl\"ucker coordinates in $\oqgkn$. \\

\begin{proposition}Let $[I]$ be a quantum Pl\"ucker coordinate in $\oqgkn$ and let $1\leq i\leq k$ and $1\leq j\leq n-k$. Then, 

(i) $\widetilde{D_{i\ast}}([I]) = -\delta(k+1-i\in I)[I]
$, and 

(ii) $\widetilde{D_{\ast j}}([I]) = \delta(k+j\in I)[I]$.
\end{proposition} 

\begin{proof}
Write $I=I_{\leq k}\sqcup I_{>k}$.  Then, with 
$A=(k+1)-(\{1,\dots,k\}\backslash I_{\leq k})$ and $B=I_{>k} -k$, we see that $[I]=[A\mid B]y$.

In case (i), 
\begin{eqnarray*}
\widetilde{D_{i\ast}}([I])
&=& 
\widetilde{D_{i\ast}}([A\mid B]y) \\
&=&
D_{i\ast}([A\mid B])y + [A\mid B]D_{i\ast}(y)\\
&=&
\delta(i\in A)[A\mid B]y-[A\mid B]y \\
&=&
( \delta(i\in A)-1)[I]
\end{eqnarray*}
Notice that $A=(k+1)-(\{1,\dots,k\}\backslash I_{\leq k})=(k+1)-(\{1,\dots,k\}\backslash I)$.

Hence, $i\in A$ if and only if $k+1-i\not\in I$. It follows that $\delta(i\in A)+\delta(k+1-i\in I) =1$. Hence, 
$\delta(i\in A)-1=-\delta(k+1-i\in I)$, and the result follows.

In case (ii), 
\begin{eqnarray*}
\widetilde{D_{\ast j}}([I]) &=& \widetilde{D_{\ast j}}([A\mid B]y) \\
&=&
 \widetilde{D_{\ast j}}([A\mid B])y+ [A\mid B]\widetilde{D_{\ast j}}(y)\\
 &=&
 D_{\ast j}([A\mid B]y + 0\\
 &=& 
 \delta(j\in B)[I]
\end{eqnarray*}
Note that $B=I_{>k} -k$. Hence, $j\in B=I_{>k} -k$ if and only if $k+j\in I$, and the result follows. 
\end{proof} 

\begin{corollary}\label{corollary-derivation-equalities}
(i) $\widetilde{D_{i\ast}} = -\widetilde{D_{k+1-i}}$, for $1\leq i\leq k$, and 
(ii) $\widetilde{D_{\ast j}}  = \widetilde{D_{k+j}}$, for $1\leq j\leq n-k$.~\\
\end{corollary} 

\begin{example}\label{example-2x4case}
For $\oqgtwofour$, we have $\widetilde{D_{1\ast}} = -\widetilde{D_2}, ~\widetilde{D_{2\ast}} = -\widetilde{D_1}$ and $\widetilde{D_{\ast 1}} = \widetilde{D_3}, 
~\widetilde{D_{\ast 2}} = \widetilde{D_4}$.~\\
\end{example} 

\begin{remark}\label{remark-D-values}
Note, for later use, that 
\[
(\widetilde{D_1}+\dots+\widetilde{D_k}+\widetilde{D_{k+1}}+\dots +\widetilde{D_{k+p}})|_{\oqmkp}= -(D_k+\dots +D_1)+D_{k+1}+\dots +D_{k+p}=0,
\]
 while 
$(\widetilde{D_1}+\dots+\widetilde{D_k}+\widetilde{D_{k+1}}+\dots +\widetilde{D_{k+p}})(y)=ky$. 
\end{remark}




\section{Adjusting derivatives} \label{section-adjusting-derivatives}

In cohomology calculations, we can adjust our original derivation by adding or subtracting derivations that arise as inner derivations. In this section, we consider what can happen when we do this, or when adjusting by column derivations. 

Let $[u]=[1,\dots,k]$ and $[w]=[n-k+1,\dots,n]$ denote the leftmost and rightmost quantum Pl\"ucker coordinates of $\oqgkn$, respectively. For any 
quantum Pl\"ucker coordinate $[I]$, set $d(I):= \# \left(I\backslash (I\cap u)\right)$ and $e(I):= \# \left(I\backslash (I\cap w)\right)$. This notation is fixed for the rest of the paper. We will use the commutation relations for $[u]$ and $[w]$ with other quantum Pl\"ucker coordinates that are described in the following lemma without comment throughout the paper.
\begin{lemma} \label{lemma-u-w-relations}
(i) $[u][I]=q^{d(I)}[I][u]$, and \\
(ii) $[w][I]=q^{-e(I)}[I][w]$.
\end{lemma}
\begin{proof} 
(i) is  proved in \cite[Lemma 3.1]{ll-aut-qg} and (ii) is proved in \cite[Lemma 1.5]{klr}.
\end{proof}

If $S= [I_{1}]\dots[I_{m}]$ is  a standard monomial, 
then set $d(S):= \sum d(I_i)$ and note that each $d(I_i)\geq 0$ with $d(I_i)=0$ if and only 
if $[I_i]=[u]$. Then, $S[u]=q^{d(S)}[u]S$, and so $Su=uS$ if and only if $d(S)=0$ (in which case 
$S=[u]^m$). \\

In any case, note that $[u]S$ is a standard monomial, as $[u]$ is the 
unique minimal quantum Pl\"ucker coordinate. \\

Recall from Section~\ref{section-basic-definitions} that $\oqgkn$ is a graded algebra with each quantum Pl\"ucker coordinate having degree one.

\begin{lemma} \label{lemma-degree-zero-term}
Let $D$ be a derivation of $\oqgkn$ and suppose that $D([I])=b_0+\dots+b_t$ is the  homogeneous decomposition of $D([I])$.  Then $b_0=0$.
\end{lemma} 

\begin{proof} Suppose that $[I]\neq [u]$ and suppose that 
$D([u]) = a_0+\dots+a_s$ is 
the homogeneous decomposition of $D([u])$. Let $d=d([I])$ and note that $d>0$ so that $q^d\neq 1$. By applying $D$ to the equation $[u][I] = q^d[I][u]$, 
we obtain
\[
D([u])[I]+[u]D([I])= q^dD([I])[u]+q^d[I]D([u]).
\]
Examination of the terms in degree one in this equation reveals that 
\[
a_0[I]+b_0[u]=q^db_0[u]+q^da_0[I],
\]
from which it follows that $a_0=q^da_0$ and $b_0=q^db_0$. Then $a_0=b_0=0$, as $q^d\neq 1$.
\end{proof} 

The following lemma is deduced from results in \cite{lln}. In the proof of the lemma, we use the notation developed in that paper without further explanation. 

\begin{lemma}\label{lemma-t-positive}
Let $v=\{1,\dots,k-1,k+1\}$ and let $[\alpha]>[v]$ in the standard order on quantum Pl\"ucker coordinates. Set $t:=|\alpha\backslash(v\cap\alpha)|$ and note  that $t>0$. Then $[v][\alpha]=q^t[\alpha][v]$ modulo $\left<u\right>$. 
\end{lemma}

\begin{proof}
There is an isomorphism 
\[
\Psi:\frac{\oqgkn}{\left<u\right>}[[v]^{-1}]\longrightarrow{\mathcal O}_{q^{-1}}(Y_\lambda))[Y^{\pm1};\sigma]
\]
where $\sigma(x_{ij})=qx_{ij}$ for the generators $x_{ij}$ of the partition subalgebra $Y_\lambda$, see 
\cite[Theorem 9.17]{lln}. If $[I\mid J]_{q^{-1}}$ is a $t\times t$ pseudo quantum minor in $Y_\lambda$ then
$Y[I\mid J]_{q^{-1}}=q^t[I\mid J]_{q^{-1}}Y$.\\

Let $\alpha\in\Pi\backslash\{u,v\}$ and set $t:=|\alpha\backslash v|$. Then
\[
\Psi(\overline{\alpha}\,\overline{v}^{-1}) = \beta[I\mid J]_{q^{-1}}
\]
for some $\beta\in\k$, and $t\times t$ pseudo quantum minor $[I\mid J]_{q^{-1}}$, by \cite[Theorem 9.17]{lln}.

Hence, 
\begin{eqnarray*}
\Psi(\overline{v}\,\overline{\alpha})
&=&
\Psi(\overline{v})\Psi(\overline{\alpha}\,\overline{v}^{-1})\Psi(\overline{v})\\
&=&
Y\beta[I\mid J]_{q^{-1}}Y\\
&=&
q^t\beta[I\mid J]_{q^{-1}}Y^2\\
&=&
q^t\Psi(\overline{\alpha}\,\overline{v}^{-1})\Psi(\overline{v}^2)\\
&=&
q^t\Psi(\overline{\alpha}\,\overline{v})
\end{eqnarray*}
and so $\overline{v}\,\overline{\alpha}=q^t\overline{\alpha}\,\overline{v}$, as required.
\end{proof}

\begin{definition}
Suppose that $a=\sum_{i=1}^s\, a_iS_i$ is an expression for an element $a\in\oqgkn$ in terms of standard monomials $S_i$ and that each $a_i$ is nonzero. Then the {\em support} of $a$, written 
$\supp(a)$, is defined to be $\{S_i\mid i=1,\dots,s\}$. 
\end{definition}

\begin{lemma}\label{lemma-no-standard-monomials}
Let $D$ be a derivation on $\oqgkn$. Then, \\
(i) there are no standard monomials of the form $[I_1]^{a_1}[I_2]^{a_2}\dots [I_m]^{a_m}$, with $I_1\neq u$ and $\sum a_i>0$, occurring in $\supp(D[u])$;\\
(ii)  there are no standard monomials of the form $[I_1]^{a_1}[I_2]^{a_2}\dots [I_m]^{a_m}$, with $I_m\neq w$ and $\sum a_i>0$, occurring in $\supp(D[w])$.
\end{lemma} 

\begin{proof} 
(i) The proof is by contradiction, so suppose that such a standard monomial exists. Without loss of generality, we may assume that $I_1=v=\{1,\dots,k-1,k+1\}$ as we are allowing the possibility that some $a_i=0$. As $u$ and $v$ differ only in one value, we know that $[u][v]=q[v][u]$. 

Apply $D$ to the equation $[u][v]=q[v][u]$ to obtain
\[
D([u])[v]+[u]D([v]) = qD([v])[u] + q[v]D([u]).
\]
Suppose that $D([u])=\sum_i\,\alpha_iS_i\, +\sum_i\,\alpha^{'}_iS^{'}_i $ where $\alpha_i\in\k$ and the $S_i,S^{'}_i $ are standard monomials such that $[u]$ does not occur in the $S_i$, but does occur in the $S^{'}_i $. Also, suppose that $D([v])=\sum_j\,\beta_jT_j$ where $\beta_j\in\k$ and the $T_i$ are standard monomials. Then
\[
\sum_i\,\alpha_iS_i[v] + \sum_i\,\alpha^{'}_iS^{'}_i [v] + \sum_j\,\beta_j[u]T_j  = \sum_j\,q\beta_jT_j[u]     +    \sum_i\,q\alpha_i[v]S_i
+ \sum_i\,q\alpha^{'}_i[v]S^{'}_i
\]
Now, $S_i[v]=q^{-t_i}[v]S_i$  modulo $\left<u\right>$ for some $t_i\geq0$, by Lemma~\ref{lemma-t-positive} (we allow zero as it might be that $S_i$ is a power of $[v]$.). 

The image of the above equation in $\oqgkn/\left<[u]\right>$ gives 
\[
\sum_i\, q^{-t_i}\alpha_i\overline{[v]}\,\overline{S_i} = \sum_i\, q\alpha_i\overline{[v]}\,\overline{S_i}
\]
or,
\[\sum_i\, (q^{-t_i}-q)\alpha_i\overline{[v]}\,\overline{S_i} = 0.
\]
Now, $\oqgkn/\left<[u]\right>$ is the quantum Schubert variety determined by the quantum Pl\"ucker coordinate $[v]$ (see \cite{lrigal} for the definition of quantum Schubert varieties in the quantum grassmannian). As such, $\oqgkn/\left<[u]\right>$ has a basis consisting of the images of the standard monomials in $\oqgkn$ that do not involve $[u]$, see \cite[Example 2.1.3]{lrigal}. Consequently, each $(q^{-t_i}-q)\alpha_i$ in the equation above is equal to zero. As at least one $\alpha_i$ is nonzero, this gives $q^{-t_i}-q=0$, for that $i$, which is a contradiction as $q$ is not a root of unity and $t_i\geq 0$.

(ii) Follows in a similar fashion. 
\end{proof}

The next corollary follows immediately from the previous lemma. \\

\begin{corollary} \label{corollary-degree-one-part}
Let  $D$ be a derivation of $\oqgkn$.  Suppose that 
$D([u])=a_1+\dots +a_s$ is the homogeneous decomposition of $D([u])$ and that 
$D([w])=b_1+\dots+b_t$ is the  homogeneous decomposition of $D([w])$. Then $a_1$ is a scalar multiple of $[u]$ and $b_1$ is a scalar multiple of $[w]$. 
\end{corollary}

In order to prove Conjecture~\ref{conjecture}, we will start by showing that we can adjust an arbitrary derivation $D$ by 
adding or subtracting derivations coming from the derivations mentioned in the conjecture (column derivations and inner derivations) so that the adjusted derivation, which we will continue to denote by $D$, satisfies $D([u])=D([w])=0$.
This will enable us to transfer the study of $D$ into  a quantum matrix problem by using the dehomogenisation equality. We show 
that we can make this adjustment in a number of steps that demonstrate how to remove standard monomials that occur in $D([u])$ (and $D([w])$).\\

As a result of Lemma~\ref{lemma-no-standard-monomials}, any standard monomial that occurs in the support of $D([u])$ must start with $[u]$. 
Similarly, any standard monomial that occurs in the support of $D([w])$ must finish with $[w]$. \\

The next lemma shows that, by adjusting $D$ by suitable inner derivations, we can remove terms of $D([u])$ that are not of the form $\alpha[u]^a$ for values of 
$a\geq 1$.\\

\begin{lemma}\label{lemma-remove-terms}
Suppose that $S=[u]^a[I_1]^{b_1}[I_2]^{b_2}\dots [I_m]^{b_m}$  is a standard monomial with  $[v]\leq [I_1]$ and $\sum_i\,b_i>0$, while $a\geq 1$. 
Let $D$ be a derivation of $\oqgkn$ and suppose that $S$ occurs in the support of $D([u])$ with nonzero scalar coefficient $\alpha$. 
Set $d:=d(S)=b_1d(I_1)+\dots+b_md(I_m)$ and set 
$z:= \frac{\alpha}{q^{-d}-1}[u]^{a-1}[I_1]^{b_1}[I_2]^{b_2}\dots [I_m]^{b_m}$.  
Also, set $D':=D-\ad_z$. 
Then 
\[
\supp(D'([u])) = \supp(D([u]))\backslash\{S\}.
\]
\end{lemma}

\begin{proof}
Note that $d>0$, so $q^{-d}-1\neq 0$. 
We calculate 
\begin{eqnarray*} 
ad_z([u]) &=& zu-uz\\
&=&
\frac{\alpha}{q^{-d}-1}\left([u]^{a-1}[I_1]^{b_1}[I_2]^{b_2}\dots [I_m]^{b_m}[u] -
[u]^a[I_1]^{b_1}[I_2]^{b_2}\dots [I_m]^{b_m}\right)\\
&=&
\frac{\alpha}{q^{-d}-1}\left(q^{-d}[u]^a[I_1]^{b_1}[I_2]^{b_2}\dots [I_m]^{b_m} -[u]^a[I_1]^{b_1}[I_2]^{b_2}\dots [I_m]^{b_m}\right)\\
&=&
\frac{\alpha}{q^{-d}-1}\left((q^d-1)[u]^a[I_1]^{b_1}[I_2]^{b_2}\dots [I_m]^{b_m}\right)\\
&=&
\alpha S
\end{eqnarray*}
It follows that $\supp(D'([u])) = \supp(D([u]))\backslash\{S\}$, as required.
\end{proof}

\begin{corollary}\label{corollary-homogeneous-terms}
Let $D$ be a derivation of $\oqgkn$. Then there is a derivation $D'$ such that $D'-D$ is a sum of inner derivations and such that the homogeneous terms of $D'([u])$ are of the form $\lambda_a[u]^a$ for $a\geq 1$ and $\alpha_a\in\k$. 
\end{corollary}

\begin{proof} Lemma~\ref{lemma-no-standard-monomials} shows that $\supp(D([u]))$ has no terms whose standard monomials do not begin with $[u]$. By using Lemma~\ref{lemma-remove-terms} an appropriate number of times, we can remove  terms in $\supp(D([u]))$ that involve $[u]$ and at least one other quantum Pl\"ucker coordinate by adjusting by suitable inner derivations. What remains is a derivation $D'$ whose support only involves terms of the form $[u]^a$. 
\end{proof}

Our next task is to show that we can adjust further, if necessary, to see that we can reduce to $a=1$ being the only possibility.\\

\begin{lemma}Suppose that $D$ is a derivation of $\oqgkn$ such that $[u]^{a-1}[w]$ occurs in $\supp(D([w]))$ with $a>1$. Then there is a derivation $D'$ of $\oqgkn$ 
such that $D'-D$ is an inner derivation 
and  $D'([u])=D([u])$  while  $\supp(D'([w]))=\supp(D([w]))\backslash [u]^{a-1}[w]$. 
\end{lemma}
\begin{proof} 
Suppose that $[u]^{a-1}[w]$ occurs in $\supp(D([w])$, say with nonzero coefficient $\beta$. Now, 
$\ad_{[u]^{a-1}}([w])=[u]^{a-1} [w]-[w] [u]^{a-1} =(1-q^{-d(w)(a-1)})[u]^{a-1} [w]$, and note that $1-q^{-d(w)(a-1)}\neq 0$, as both $d(w)$ and $a-1$ are nonzero and $q$ is not a root of unity. 
Set $D':= D-\beta(1-q^{-d(w)(a-1)})^{-1}\ad_{[u]^{a-1}}$. Then $D'([u])=D([u])$, as $\ad_{[u]^{a-1}}([u])=0$. Also, 
$D'([w])= D([w])-\beta(1-q^{-d(w)(a-1)})^{-1}\ad_{[u]^{a-1}}([w])=D([w])-\beta[u]^{a-1}[w]$, so that 
$\supp(D'([w]))=\supp(D([w]))\backslash [u]^{a-1}[w]$, as required. 
\end{proof}

The following corollary now follows by applying the previous lemma an appropriate number of times.

\begin{corollary}\label{corollary-no-terms-in-dw}
Let $D$ be a derivation of $\oqgkn$. Then there is a derivation $D'$ of $\oqgkn$ such that $D' - D$ is a sum of inner derivations, $D'([u])=D([u])$ and such that there are no terms of the form $[u]^{a-1}[w]$ with $a>1$ occuring in $\supp(D([w]))$. 
\end{corollary}

\begin{lemma}\label{lemma-only-u1}
Let $D$ be a derivation of $\oqgkn$ such that $D([u])=\sum \lambda_i[u]^i$ for some $\lambda_i\in\k$ and suppose that no terms of the form $[u]^{a-1}[w]$ with $a>1$ occur in $\supp(D([w]))$. Then $D([u])=\lambda_1[u]$.
\end{lemma}

\begin{proof} 
Suppose that $D([w])=\sum\beta_iS_i$ for some standard monomials $S_i$ and $0\neq \beta_i\in\k$ and note that there is no $S_i$ such that 
$S_i=[u]^{a-1}[w]$ for any $a>1$, by assumption. Apply $D$ to the equation $[u][w]=q^{d(w)}[w][u]$ to obtain 
\[
D([u])[w]+[u]D([w])=q^{d(w)}D([w])[u]+q^{d(w)}[w]D([u]).
\]
Hence,
\begin{align*}
\sum \lambda_i[u]^i[w] + \sum\beta_i[u]S_i &= \sum q^{d(w)}\beta_iS_i[u] + \sum q^{d(w)}\lambda_i[w][u]^i\end{align*}
and so 
\begin{align*}
\sum \lambda_i[u]^i[w] + \sum\beta_i[u]S_i &=\sum q^{d(w)-d(S_i)}\beta_i[u]S_i + \sum q^{d(w)-i\cdot d(w)}\lambda_i[u]^i[w] .
\end{align*}
The terms in this equation are all scalar multiples of standard monomials. Consider the occurences of the standard monomial $[u]^a[w]$ for a given $a>1$. 
If $[u]S_i=[u]^{a}[w]$ then $S_i=[u]^{a-1}[w]$, which does not occur, by assumption. Hence, 
the second term on the left side of this equation and the first term on the right side do not contain $[u]^{a}[w]$. It follows that $\lambda_a[u]^a[w]=q^{(1-a)d(w)}\lambda_a[u]^a[w] $, and this forces 
$\lambda_a(1-q^{(1-a)d(w)}) =0$. Now, $(1-a)d(w)\neq 0$, as $a\neq 1$ so $(1-q^{(1-a)d(w)})\neq 0$. Hence, $\lambda_a=0$. As this is true for all $a>1$, we obtain the required result. 
\end{proof} 

Recall from Lemma~\ref{lemma-column-derivations} that $D_1$ is the column derivation defined by $D_{1}([I])=\delta(1\in I)[I]$ for each quantum Pl\"ucker coordinate $[I]$.

\begin{corollary} Let $D$ be a derivation of $\oqgkn$ such that $D([u])=\lambda[u]$ for some $\lambda\in\k$ and suppose that no terms of the form $[u]^{a-1}[w]$ with $a>1$ occur in $\supp(D([w]))$. Then there is a derivation $D'$ of $D$ such that the following hold:\\
(i) $D'([u])=0$;\\
(ii) $D'-D=\lambda D_1$;\\
(iii) there are no terms of the form $[u]^{a-1}[w]$ with $a>1$ that occur in $\supp(D'([w]))$.
\end{corollary}
\begin{proof} Note that $D_1([u]) = D_1([1,\dots,k])=[u]$ and $D_1([w])=D_1([n-k+1,\dots,n])=0$, as $1<n-k+1$.
Set $D'=D-\lambda D_1$ so that $D'([u])=D([u])-\lambda D_1([u]) =\lambda[u]-\lambda[u]=0$. Also, $D'([w])=D([w])-\lambda D_1([w]) =
D([w])$; so no terms of the form $[u]^{a-1}[w]$ with $a>1$ occur in $\supp(D([w]))$.
\end{proof}

\begin{lemma}\label{lemma-no-terms-in-w}
Let $D$ be a derivation of $\oqgkn$ with the following properties:\\
(i) $D([u])=0$,\\
(ii) $\supp(D([w]))$ contains no term of the form $[u]^{a-1}[w]$ for $a>1$.\\
Then $D([w])=\alpha[w]$ for some $\alpha\in\k$. 
\end{lemma}


\begin{proof}
Recall, from Lemma~\ref{lemma-no-standard-monomials}, that $[w]$ must occur in any standard monomial contained in $\supp(D([w]))$. Suppose that $S=[u]^{a}[I_1]^{a_1}\dots [I_m]^{a_m}[w]^b\in \supp(D([w]))$ with $u<I_1$ and $I_m<w$ while  $a,a_i\geq 0, b>0$. Suppose  that $S$ occurs in $D([w])$ with nonzero coefficient $\alpha\in\k$.
Note that $d(I_i)>0$ for each $i$.
Apply $D$ to the equation $[u][w]=q^{d(w)}[w][u]$, remembering that $D([u])=0$, to obtain
\[
[u]D([w])=q^{d(w)}D([w])[u] 
\]
Examination of the standard monomials of the form $[u]S$ in this equation reveals that  
\begin{eqnarray*}
\alpha[u]^{a+1}[I_1]^{a_1}\dots [I_m]^{a_m}[w]^b 
&=&
\alpha q^{d(w)}[u]^a[I_1]^{a_1}\dots [I_m]^{a_m}[w]^b[u]\\
&=&
\alpha q^{d(w)-\left(d(w)b+a_1d(I_1)+\dots+a_md(I_m)\right)}[u]^{a+1}[I_1]^{a_1}\dots [I_m]^{a_m}[w]^b.
\end{eqnarray*}
As $q$ is not a root of unity, the only possibility is that the power of $q$ on the right hand side is $q^0$, and this is only possible for $b=1$ and $a_1=\dots=a_m=0$. Thus the only possible terms in $\supp(D[w])$ are of the form $[u]^a[w]$. Taking into account condition (ii) in the statement of the lemma, we see that $a>0$ is not allowed, so $D[w]=\alpha[w]$, as required. 
\end{proof} 


Recall from Lemma~\ref{lemma-column-derivations} that there are column derivations $D_i$, for $i=1,\dots,n$ such that $D_{i}([I])=\delta(i\in I)[I]$ for each quantum Pl\"ucker coordinate $[I]$. The results of this section are summarised in the following proposition. 
\begin{proposition}\label{proposition-adjusting-derivations}
Let $D$ be a derivation of $\oqgkn$. Then there is a derivation $D'$ of $\oqgkn$ with $D'([u])=D'([w])=0$ and such that 
$(D'-D)$ is a linear combination of derivations of the form $\ad_z$, with $z\in\oqgkn$, and column derivations $D_i$ for $i=1,\dots,n$. 
\end{proposition}

\begin{proof} We know the following facts hold for any derivation $D$ of $\oqgkn$ and so will hold for the any derivation that occurs when we adjust a given derivation by adding or subtracting inner derivations and scalar multiples of the $D_i$: (i) the degree zero parts of $D([u])$ and $D([w])$ are both zero, see Lemma~\ref{lemma-degree-zero-term}; (ii) the degree one part of $D([u])$ is a scalar multiple of $[u]$ and, similarly, the degree one part of $D([w])$ is a scalar multiple of $[w]$, see Corollary~\ref{corollary-degree-one-part}; (iii) 
any standard monomial occurring in the support of $D([u])$ must start with at least one occurrence of $[u]$ and, similarly, any standard monomial occurring in the support of $D([w])$ must end with at least one occurrence of $[w]$, see Lemma~\ref{lemma-no-standard-monomials}.

Let $D$ be an arbitrary derivation of $\oqgkn$. By Corollary~\ref{corollary-homogeneous-terms} there is a derivation $D^{(1)}$ of $\oqgkn$ such that $D^{(1)}-D$ is a sum of inner derivations and such that the homogeneous terms of $D^{(1)}([u])$ are of the form $\lambda_a[u]^a$ for $a\geq 1$ and $\alpha_a\in\k$. 

By Corollary~\ref{corollary-no-terms-in-dw} there is a derivation $D^{(2)}$ of $\oqgkn$ such that 
$D^{(2)}-D^{(1)}$ is a sum of inner derivations, 
the homogeneous terms of $D^{(2)}([u])=D^{(1)}([u])$ are of the form $\lambda_a[u]^a$ for $a\geq 1$ and $\alpha_a\in\k$ and such that there are no terms of the form $[u]^{a-1}[w]$ with $a>1$ occur in $\supp(D^{(2)}([w]))$. 
By Lemma~\ref{lemma-only-u1}, $D^{(2)}([u])=\lambda[u]$ for some $\lambda\in\k$. 

Set $D^{(3)}:=D^{(2)}-\lambda D_1$. Then $D^{(3)}([u])=0$ while $D^{(3)}([w])=D^{(2)}([w])-\lambda D_1([w]) = D^{(2)}([w])$, as 
$ D_1([w])=0$. Hence,  there are no terms of the form $[u]^{a-1}[w]$ with $a>1$ in $D^{(3)}([w])$.  It follows from Lemma~\ref{lemma-no-terms-in-w} that $D^{(3)}([w])=\alpha[w]$ for some $\alpha\in\k$. 

Finally, set $D':=D^{(3)}-\alpha D_n$. Then $D'([u])=D^{(3)}[u]=0$, as $D_n([u])=0$ and $D'([w])=D^{(3)}([w])-\alpha D_n([w])
=0$. The passage from $D$ to $D'$ via $D^{(1)}, D^{(2)}, D^{(3)}$ only involves adjustments by adding or subtracting derivations of the form $\ad_z$, with $z\in\oqgkn$, and $D_i$ for $i=1,\dots,n$ at each stage, so the required result follows.
\end{proof}




\section{Transferring derivations of $\oqgkn$ to $\oqmkp$} \label{section-transferring-derivations}

Throughout this section, we assume that $2k\leq n$.

Recall the dehomogenisation equality from Section~\ref{section-dehomogenisation-equality}
\[
T=\oqgkn[[u]^{-1}]=\oqmkp[y,y^{-1}; \sigma].
\]
Given a derivation $D$ of $\oqgkn$ with $D([u])=D([w])=0$, we may extend $D$ to $T$ by setting $D([u]^{-1})=0$ and then transfer, via the dehomogenisation equality, to $\oqmkp[y,y^{-1}; \sigma]$. We then know that $D(y)=D([u])=0$. We retain the notation $D$ for this extension to $T$. \\

Recall that in Section~\ref{section-dehomogenisation-equality} we set $R:=\oqmkp$ where $p=n-k$. The quantum matrix generators $x_{ij}$ of $R$ were defined in Section~\ref{section-dehomogenisation-equality}: 
\[
x_{ij}:=[1\dots,\widehat{k+1-i},\dots k, j+k][u]^{-1}\in T.
\]

Our aim in this section is to show that $D(R)\subseteq R$ for such a derivation $D$. We will use a pair of gradings of $T$ that were developed in \cite[Section 6]{ll-aut-qg} to discuss a similar result for certain automorphisms of $\oqgkn$. \\

As $2k\leq n$, we know that $k\leq n-k=p$ and so $R$ has at least $k$ columns and the 
quantum minor 
$[I\mid J]:= [1\dots k\mid p+1-k, \dots, p]\in R$ is defined (we are using all the rows 
of $R=\oqmkp$ and the last 
$k$ columns). As noted in \cite[Lemma 6.1]{ll-aut-qg}, $x_{ij}[I\mid J]=q[I\mid J]x_{ij}$ when $j<p+1-k$ while $x_{ij}[I\mid J]=[I\mid J]x_{ij}$ when $j\geq p+1-k$. As a consequence, $[I\mid J]$ is a normal element in $R$ and also in $T$. \\

\begin{lemma}\label{lemma-d-of-determinant-is-zero}
Let $D$ be a derivation of $\oqgkn$, where $2k\leq n$, with $D([u])=D([w])=0$. Let $[I\mid J]$ be defined as in the previous paragraph. Then $D([I\mid J])=0$.
\end{lemma}

\begin{proof} The discussion at the end of Section~\ref{section-dehomogenisation-equality} shows that 
$$[I\mid J] = [p+1\dots n][1\dots k]^{-1}=[w][u]^{-1};$$ 
and so 
$D([I\mid J])=D([w][u]^{-1})=0$.
\end{proof} 

Also, we can calculate how $[I\mid J]$ commutes with $[u]=[1\dots k]$. Note that $k<n-k+1=p+1$, as $2k\leq n$. 
Thus the index sets $\{1,\dots, k\}$ and $\{p+1,\dots, n\}$ do not overlap, and 
\[
[u][I\mid J]= [u][w][u]^{-1} = q^k[w][u][u]^{-1}
= q^k[w][u]^{-1}[u] = q^k[I\mid J]\,[u], 
\]
where the second equality comes from Lemma~\ref{lemma-u-w-relations}.  \\

The two gradings that were used in \cite[Section 6]{ll-aut-qg} are defined by considering how elements of $T$ commute with $y=[u]$ and with $[I\mid J]$.\\

 We set $T_i:=\{a\in T\mid yay^{-1}=q^ia\}$ and $T^{(i)}:=\{a\in T\mid 
[I\mid J]a[I\mid J]^{-1}=q^{-i}a\}$. 

\begin{lemma} 
(i) $T=\bigoplus_{i=1}^{\infty}T_i$\\
(ii) $T=\bigoplus_{i=\mz}T^{(i)}$\\
(iii) $(T^{(0)}\cup T^{(1)})\cap T_1\subseteq \oqmkp=R$
\end{lemma} 

\begin{proof}
These results are established in \cite[Lemma 6.2(i), Lemma 6.3(i), Lemma 6.4]{ll-aut-qg}
\end{proof}


\begin{theorem}\label{theorem-dr-in-r}
Suppose that $2k\leq n$ and that $D$ is a derivation of $\oqgkn$ such that $D([u])=D([w])=0$. Then $D(R)\subseteq R$. 
\end{theorem}

\begin{proof} 

It is enough to show that $D(x_{ij})\in R$ for each generator $x_{ij}$ of $R$. 

Note that $x_{ij}\in (T^{(0)}\cup T^{(1)})\cap T_1$. This claim follows from  the commutation rules given in \cite[Lemma 6.1]{ll-aut-qg} and the fact that $yx_{ij}=qx_{ij}y$.

Let $a\in T_1$. Then $ya =qay$. Apply $D$ to this equation, noting that $D(y)=0$, to obtain $yD(a)=qD(a)y$ so that $D(T_1)\subseteq T_1$. Similar calculations, using the fact that $D([I\mid J])=0$  show that $D(T^{(0)})\subseteq T^{(0)}$ and $D(T^{(1)})\subseteq T^{(1)}$. It follows that $D(x_{ij})\subseteq D((T^{(0)}\cup T^{(1)})\cap T_1)\subseteq (T^{(0)}\cup T^{(1)})\cap T_1 \subseteq R$.

As $D$ takes each generator $x_{ij}$ of $R$ into $R$, we see that 
$D(R)\subseteq R$, as required.
\end{proof} 




 
\section{The main theorem} 

Recall that we are  assuming  that $\k$ is a field of characteristic zero and that $q\in\k$ is a nonzero element that is not a root of unity. We are also assuming that $k\neq 1$, as in this case the quantum grassmannian is a quantum affine space, where the results are known, see \cite{ac}. As $\oq(G(n-1,n))\cong \oq(G(1,n))$, we exclude this case as well. Thus, we are assuming that $2\leq k\leq n-2$.\\

In this section, we prove Conjecture~\ref{conjecture}. The proof proceeds by first analysing the case where $2k\leq n$. The general case is then obtained by using the isomorphism $\oqgkn\cong\oqgnminuskn$. In order to avoid breaking the flow of the main result, we relegate to  an appendix a discussion concerning homogeneous derivations on non-square quantum matrices that we use in obtaining the truth of the conjecture in the case where $2k<n$.\\

In this section, in order to make reading easier,  we will use $\delta$ to denote an arbitrary derivative. \\

\subsection{The case where $2k\leq n$}

In this subsection we consider $\oqgkn$ in the case that $2k\leq n$. In this case, the dehomogenisation equality is: 
\[
\oqgkn[[u]^{-1}]=T=\oqmkp[y, y^{-1};\sigma], 
\]
where $p=n-k$. We set $R:=\oqmkp$. \\

When  $2k=n$ so that $R=\oqmkk$, it is well known that the centre of $R$ is $\k[D_q]$, where $D_q$ is the quantum determinant of $\oqmkk$. It is also well known that when $2k<n$, so that $R$ is non-square, the centre of $R$ is $K$. We also recall that the centre of $\oqgkn$ is always reduced to scalars. This follows easily from the basis of standard monomials by first observing that an element is central in $\oqgkn$ if and only if all the standard monomials in its support are central, and next by noting that there are no nontrivial central standard monomials since the only standard monomial commuting with both $[u]$ and $[w]$ are scalars by Lemma \ref{lemma-u-w-relations}.

\begin{proposition}\label{proposition-main-2k<=n}
Assume that $2k\leq n$. Then any derivation $\delta$ of $\oqgkn$, is equal, modulo inner derivations, to a linear combination of $D_1,\dots, D_n$. Furthermore, these $n$ derivations are linearly independent modulo the inner derivations.
\end{proposition} 

\begin{proof} 
We use the same notation $\delta$ for the extension to $T$. After possibly adjusting $\delta$ by inner derivations and linear combinations of $D_1,\dots, D_n$ we may assume that  $\delta(y)=\delta([u])=\delta([w])=0$, by Proposition~\ref{proposition-adjusting-derivations}, and then $\delta(R)\subseteq R$, by Theorem~\ref{theorem-dr-in-r}. Apply $\delta$ to the equation $yx_{ij}=qx_{ij}y$ to obtain
\[
y\delta(x_{ij})=q\delta(x_{ij})y.
\]
Given this equation and the fact that $\delta(R)\subseteq R$, we conclude that $\delta(x_{ij})$ is homogeneous of degree one. \\

In order to prove the first claim, we consider the two cases (i) $2k=n$ and (ii) $2k<n$ separately in order to show the {\bf subclaim} that  
$\delta|_R$ can be written as a linear combination of the row and column derivations $D_{i*}, D_{*j}$ of the quantum matrix algebra $R$ introduced in Section~\ref{section-via-dhom}. \\

{\bf Subclaim: Case (i).} First, suppose that $2k=n$, so that $k=n-k$. In this case, $R=\oqmkk$, and so $R$ is a square quantum matrix algebra and the centre of $R$ is $\k[D_q]$, where $D_q$ is the quantum determinant of $\oqmkk$.\\

By \cite[Theorem 2.9]{ll-qmatrix-ders}, there are polynomials $P_1,\dots,P_k,Q_1,\dots, Q_k\in\k[D_q]$ and an element $z\in R$ such that 
\begin{eqnarray*}
\delta|_R=\ad_z +\sum_{i=1}^k P_iD_{i\ast} +\sum_{j=1}^{n-k} Q_jD_{\ast j},
\end{eqnarray*}
where $D_{i*}, D_{*j}$ are the row and column derivations of $R$  introduced in Section~\ref{section-via-dhom}.\\

Let $a_i$ be the constant term in $P_i$ and $b_j$ be the constant term in $Q_j$. Then 
\begin{eqnarray*}
\lefteqn{\delta(x_{rs}) -\left(\sum_{i=1}^{k} a_iD_{i\ast}(x_{rs})+\sum_{j=1}^{n-k} b_jD_{\ast j}(x_{rs})\right)}\\
&=&zx_{rs}-x_{rs}z + \sum_{i=1}^{k} (P_i-a_i)D_{i\ast}(x_{rs}) +\sum_{j=1}^{n-k} (Q_j-b_j)D_{\ast j}(x_{rs})
\end{eqnarray*}
for all $r,s$. The terms on the left side of this equation all have degree one, whereas the terms on the
right hand side have degree greater than one, because $zx_{rs}-x_{rs}z$ has no degree zero or degree one terms.
\\

It follows that both sides are zero, and so $\delta|_R=\sum_{i=1}^{k} a_iD_{i\ast}+\sum_{j=1}^{n-k} b_jD_{\ast j}$, which establishes the subclaim in the case that $2k=n$.\\

{\bf Subclaim: Case (ii).} Next, suppose that $2k<n$. In this case, $R=\oqmkp$, where $p=n-k>k$ and so $R$ is a non-square quantum matrix algebra with more columns than rows, and the centre of $R$ is $\k$. 
The proof of this case is substantially more complicated than that of Case (i) due to the fact that  \cite[Theorem 2.9]{ll-qmatrix-ders} only covers derivations for square quantum matrices. To avoid disturbing the flow of the  proof of this Proposition, the proof of this subclaim is treated in the appendix and finally  established in  Proposition~\ref{proposition-non-square}.\\

Having established the subclaim, we revert to the condition that $2k\leq n$. \\

By using Corollary~\ref{corollary-derivation-equalities} we see that 
\[
\delta|_R = \sum_{i=1}^{k} a_iD_{i\ast}+\sum_{j=1}^{n-k} b_jD_{\ast j}=(-\sum_{i=1}^k a_i\widetilde{D_i}
+\sum_{j=1}^{n-k} b_j\widetilde{D_{k+j}})|_R\,.
\] Set $\widetilde{\delta}=-\sum_{i=1}^k a_i\widetilde{D_i}+\sum_{j=1}^{n-k} b_j\widetilde{D_{k+j}}$, so that $\delta|_R=\widetilde{\delta}|_R$. Note that 
$\widetilde{\delta}(y)=-(\sum_{i=1}^{k} a_i)y$, as $\widetilde{D_i}(y)=y$ for $i=1,\dots,k$, while 
$\widetilde{D_{k+j}}(y)=0$ for $j=1,\dots,n-k$. Recall from Remark~\ref{remark-D-values} that 
$\sum_{i=1}^{n} \widetilde{D_i}$ acts trivially on $R$, while 
$\sum_{i=1}^{n} \widetilde{D_i}(y)=ky$. \\

Set $\widehat{\delta}:= \left(\frac{1}{k}\sum_{i=1}^{k} a_i\right)\left(\sum_{i=1}^{n} \widetilde{D_i}\right)$.
Then, for all $r$ and $s$, we have $\left(\widetilde{\delta}+\widehat{\delta}\right)(x_{rs}) =\widetilde{\delta}(x_{rs})+\widehat{\delta}(x_{rs})= \widetilde{\delta}(x_{rs})+0=\delta(x_{rs})$, while 
$\left(\widetilde{\delta}+\widehat{\delta}\right)(y) =
\widetilde{\delta}(y)+\widehat{\delta}(y)=-(\sum a_i)y+\left(\frac{1}{k}\sum a_i\right)(ky)=0=\delta(y)$.
As $\delta$ and $\widetilde{\delta}+\widehat{\delta}$ agree on the generating set $x_{ij},y$, they are equal as derivatives. \\

For the proof of the second part, suppose that 
\[
\ad_z+\sum_{i=1}^n a_iD_i = 0
\]
for some $z\in\oqgkn$ and $a_i\in\k$. Thus, $\ad_z([I]) +\sum_{i=1}^n a_i\delta(i\in I)[I]=0$ for each quantum Pl\"ucker coordinate $[I]\in\oqgkn$. The first term has no components in degree one and the other terms are all in degree one, so we deduce that $\ad_z([I])=0$ for each quantum Pl\"ucker coordinate $[I]\in\oqgkn$ so that $\ad_z=0$. Thus $\sum_{i=1}^{n}a_iD_i=0$.\\

For $r=1,\dots, n+1-k$ set $[I_r]:=[1,\dots,k-1,k-1+r]$, and observe that 
\[
0=\sum_{i=1}^n a_iD_i[I_r]=\left((a_1+\dots+a_{k-1})+a_{k-1+r}\right)[I_r] 
\]
Thus, 
$(a_1+\dots+a_{k-1})+a_{k-1+r}=0$ for each of these values of $r$. It follows that $a_k=a_{k+1}=\dots=a_n$. In a similar manner, set $[J_r]:=[n-k+1-r,n-k+2,\dots,n]$ for $r=0,\dots,n-k$ to observe that $a_{n-k+1-r} +(a_{n-k+2}+\dots+a_n)=0$ for these values of $r$. It follows that $a_1=\dots=a_{n-k+1}$. These two ranges of values must overlap, or else $n-k+1<k$ so that $n+1<2k\leq n$, a contradiction. Thus, $a_1=a_2=\dots=a_n$ and from this and Corollary \ref{corollary-sum=identity} it follows that each $a_i=0$. 
\end{proof}

\subsection{The general case} \label{subsection-general-case}
We have now proved our conjecture for $\oqgkn$ in the case where $2k\leq n$. In order to remove this restriction, we use the fact that $\oqgkn\cong\oqgnminuskn$, see, for example, \cite[Proposition 3.1]{ll-aut-qg}. We will use $\overline{D}$ to distinguish derivations of $\oqgnminuskn$ from derivations of $\oqgkn$. Let $\psi: \oqgkn\longrightarrow\oqgnminuskn$ be the automorphism of \cite[Proposition 3.1]{ll-aut-qg}; so that 
$\psi([I])=[w_0(\,\widehat{I}\,)]$ for each quantum Pl\"ucker coordinate $[I]$ of $\oqgkn$, where $\widehat{I}:=\{1, \dots , n\} \setminus I$ and $w_0$ is the longest element of the symmetric group $S_n$. Let $D$ be a derivation of $\oqgkn$.   
It is easy to check that $\psi D\psi^{-1}$ is a derivation of 
$\oqgnminuskn$. Similarly, if $\overline{D}$ is a derivation of $\oqgnminuskn$ then 
$\psi^{-1}\overline{D}\psi$ is a derivation of $\oqgkn$.\\

Recall that we have the derivations $D_i$ of $\oqgkn$ for $i=1,\dots, n$, with $D_i([I])=\delta(i\in I)[I]$, and similarly we have derivations $\overline{D_j}$ of $\oqgnminuskn$ for $j=1,\dots, n$. \\

For each $i$, we need to see how $\psi D_i\psi^{-1}$ acts as a derivation on $\oqgnminuskn$ in terms of the $\overline{D_j}$. \\

Note from Corollary~\ref{corollary-sum=identity} that 
$\frac{1}{n-k}\left(\sum_{j=1}^n  \overline{D_j}\right)([J])=[J]$ for each quantum Pl\"ucker coordinate $[J]\in\oqgnminuskn$

\begin{lemma}
\[
\psi D_i\psi^{-1}=\frac{1}{n-k}\left(\sum_{j=1}^n  \overline{D_j}\right)~ -~\overline{D}_{w_0(i)}
\]
\end{lemma}

\begin{proof}
 Let $[J]$ be a quantum Pl\"ucker coordinate in $\oqgnminuskn$ and suppose that $[J]=\psi([I])=[w_0(\widehat{I})]$ for a quantum Pl\"ucker coordinate $[I]$ of $\oqgkn$. 
 
 Before we do the calculation of $\psi D_i\psi^{-1}$, note the following evaluation of a truth function:
\[
\delta(i\in I)=1-\delta(i\in\widehat{I}) = 
1-\delta(w_0(i)\in w_0(\widehat{I})) = 
1-\delta(w_0(i)\in J).
\]

We obtain
\begin{eqnarray*}
\psi D_i\psi^{-1}([J])&=&\psi D_i([I])=\psi(\delta(i\in I)[I])=
\delta(i\in I)[J]\\
&=&
(1-\delta(w_0(i)\in J))[J]=[J]-\overline{D}_{w_0(i)}([J])\\
&=&
\left\{\frac{1}{n-k}\left(\sum_{j=1}^n  \overline{D_j}\right)~ -~\overline{D}_{w_0(i)}\right\}([J]),
\end{eqnarray*}
as required.
\end{proof}

We can now obtain our main theorem without any restriction other than $1<k<n-1$. Given that we have proved the conjecture in the case that $2k\leq n$,  it is enough to prove the result for $\oqgnminuskn$ when $2k\leq n$.

\begin{proposition}\label{proposition-main-(n-k)-case}
Assume that   $2k\leq n$. In $\oqgnminuskn$ any derivation is equal, modulo inner derivations, to a linear combination of $\overline{D_1},\dots, \overline{D_n}$. Furthermore, these $n$ derivations are linearly independent modulo the inner derivations.
\end{proposition} 

\begin{proof}
Let $\overline{D}$ be a derivation on $\oqgnminuskn$. Then $\psi^{-1}\overline{D}\psi$ is a derivation on $\oqgkn$. Hence,
\[
\psi^{-1}\overline{D}\psi = \ad_z+ \sum_{i=1}^n a_iD_i
\]
for some $z\in\oqgkn$ and $a_i\in\k$, by Proposition~\ref{proposition-main-2k<=n}.

Therefore,
\[
\overline{D}=\ad_{\psi(z)}+\sum_{i=1}^n a_i\psi D_i\psi^{-1}
=
\ad_{\psi(z)}+\sum_{i=1}^n a_i\left\{\frac{1}{n-k}\left(\sum_{j=1}^n  \overline{D_j}\right)~ -~\overline{D}_{w_0(i)}\right\}, 
\]
and the first claim follows. 

The proof of the second part follows in the same way as for the second part of Proposition~\ref{proposition-main-2k<=n}.
\end{proof}

We are now ready to state and prove our main result.

\begin{theorem} Let $2\leq k\leq n-2$. Then any derivation of $\oqgkn$  is equal, modulo inner derivations, to a linear combination of $D_1,\dots, D_n$. Furthermore, these $n$ derivations are linearly independent modulo the inner derivations.
\end{theorem} 

\begin{proof}
The result in the case that $2k\leq n$ has been established in Proposition~\ref{proposition-main-2k<=n}; so suppose that $2k>n$. Set $k':=n-k$ then $2k'<n$ and $n-k'=k$. As $2k'<n$,  the result holds for $\oqgkdashn$. It then follows that 
the result holds for $\oqgkn=\oqgnminuskdash$, by using Proposition~\ref{proposition-main-(n-k)-case}.
\end{proof}

~\\

Recall that the Hochschild cohomology group in degree one of a ring $R$ denoted by ${\rm HH}^1(R)$, is defined by:
\[
{\rm HH}^1(R):={\rm Der}(R)/{\rm InnDer}(R),
\]
where ${\rm InnDer}(R):=\{ad_z\mid z\in R\}$ is the Lie algebra of inner derivations of $R$. It is well known that ${\rm HH}^1(R)$ is a module over ${\rm HH}^0(R):=Z(R)$.\\

The following corollary is immediate from the above theorem. 

\begin{corollary} Let $2\leq k\leq n-2$. 
The first Hochschild cohomology group of the quantum grassmannian, ${\rm HH}^1(\oqgkn)$ is an $n$-dimensional vector space over $\k$ with basis (the cosets of) $D_1,\dots,D_n$.

\end{corollary}




\appendix

\begin{section}{Derivations on non-square quantum matrices} \label{appendix-a} 

In this appendix, we prove Case (ii) of the Subclaim in the proof of Proposition~\ref{proposition-main-2k<=n}. To be more specific, 
in Case (ii) of the Subclaim, we are dealing with a derivation of $\oq(M(k,p))$, where $k<p$, that arises, via the dehomogenisation equality,  from a derivation $D$ of $\oqgkn$ that has the following properties: (i) $D([u])=D([w])=0$; (ii) $D$ is a homogeneous derivative of $\oq(M(k,p))$. As a consequence of Lemma~\ref{lemma-d-of-determinant-is-zero}, the first condition  implies that the derivative $D$ acts trivially on the rightmost quantum minor of $\oq(M(k,p))$. 

Hence, with a change of notation, throughout this appendix, we assume that we are considering derivations on $\oqmmn$ where $m<n$ 
and that we have a derivation $D$ acting on $\oqmmn$ with the following properties:

(i) the derivation $D$ is homogeneous; that is, all terms appearing non-trivially in $D(x_{rs})$ have degree one.
 
(ii) $D([1,\dots,m\mid n+1-m,\dots,n])=0$.\\

The aim in this appendix is to show that such a derivation can be written as a linear combination of the row and column derivations $D_{i*}$ and $D_{*j}$ that were introduced in Section~\ref{section-via-dhom}. With this in mind, we fix the following notation.

\begin{notation}\label{appendix-notation} Throughout the appendix, $B$ denotes the quantum matrix subalgebra of $\oqmmn$ generated by $x_{rs}$ in the first $n-m$ columns of $\oqmmn$  and $C$ denotes the square quantum matrix subalgebra of $\oqmmn$ generated by the $x_{rs}$ in the final $m$ columns. The quantum determinant of $C$ is $[I\mid J]:=[1,\dots,m\mid n-m+1,\dots,n]$ and $[I\mid J]$ is in the centre of $C$. 
\end{notation}

 \subsection{Action of derivations on first $m$ columns of $\oqmmn$} 

\begin{lemma}\label{lemma-db-in-b} Use Notation~\ref{appendix-notation}. Let $D$ be a derivation on $\oqmmn$ such that $D(x_{ij})$ is homogeneous of degree one for all $x_{ij}$ and suppose that $D([I\mid J])=0$. Then $D(B)\subseteq B$
\end{lemma}

\begin{proof} It is enough to show that $D(x_{ij})\in B$ for $j\leq m$. 
For such an $x_{ij}$ suppose that  $D(x_{ij})=\sum_{\substack{r\leq m\\s\leq n}} a_{rs}x_{rs}$ with $a_{rs} \in K$. Now, 
$x_{ij}[I\mid J]=q[I\mid J]x_{ij}$, by \cite[Lemma 6.1(ii)]{ll-aut-qg}.
Apply $D$ to this equation, noting that $D([I\mid J])=0$, to obtain
\[
\sum_{\substack{1\leq r\leq m\\1\leq s\leq n}}a_{rs}x_{rs}[I\mid J]=q[I\mid J] \sum_{\substack{1\leq r\leq m\\1\leq s\leq n}}a_{rs}x_{rs}.
\]
As $x_{rs}[I\mid J]=q[I\mid J]x_{rs}$ when $s\leq m$, and   $x_{rs}[I\mid J]=[I\mid J]x_{rs}$ for $s>m$, this gives 
\[ 
\sum_{r\leq m\atop s>m} (1-q)a_{rs}x_{rs}[I\mid J]=0
\]
and it follows that $a_{rs}=0$ when $s>m$, so that $D(x_{ij})\in B$, as required. 

\end{proof} 




\subsection{We can adjust $D$ so that  $D$ is  trivial on $C$}

\begin{lemma}\label{lemma-dc-in-c}
Use Notation~\ref{appendix-notation}. Let $D$ be a derivation on $\oqmmn$ such that $D(x_{ik})$ is homogeneous of degree one for all $x_{ik}$ and suppose that $D([I\mid J])=0$. Then $D(C)\subseteq C$. 
\end{lemma}

\begin{proof} 
This is proved in a similar manner to the proof of Lemma \ref{lemma-db-in-b}, using the fact that $x_{ik}$ commutes with $[I\mid J]$ for $x_{ik}$ in $C$.

 
 
 \end{proof}
 We now show that we can adjust $D$ by row and column derivations so that the adjusted derivation acts trivially on $C$.

  \begin{lemma}\label{lemma-linear-combination}
  Use Notation~\ref{appendix-notation}. Let $D$ be a derivation on $\oqmmn$ such that $D(x_{rs})$ is homogeneous of degree one for all $x_{rs}$ and suppose that $D([I\mid J])=0$. There is a homogeneous derivation $D'$ of $\oqmmn$ with $D'|_C=0$ and $D'(B)\subseteq B$ such that $D'-D$ is a linear combination of $D_{i*}$, for $i=1,\dots,m$, and $D_{*j}$, for $n-m+1\leq j\leq n$, the row and column derivations of $\oqmmn$  introduced in Section~\ref{section-via-dhom}.
  \end{lemma}

  \begin{proof} 
 Note that $D(C)\subseteq C$, by Lemma~\ref{lemma-dc-in-c}.
    Set $Y:=[I\mid J]$, the quantum determinant of $C$.
 By \cite[Theorem 2.9]{ll-qmatrix-ders}, there are polynomials $P_1,\dots,P_m,Q_{n-m+1},\dots, Q_n\in\k[Y]$ and an element 
 $z\in C$ such that 
\begin{eqnarray*}
D|_C=\ad_z +\sum_{i=1}^m P_iD_{i\ast}|_C +\sum_{j=n-m+1}^n Q_jD_{\ast j}|_C.
\end{eqnarray*}
 Let $s\geq n-m+1$ so that $x_{rs}\in C$.
Then 
\begin{eqnarray*}
D|_C(x_{rs}) &=& \ad_z(x_{rs}) +  \sum_{i=1}^m P_iD_{i\ast}|_C(x_{rs}) +\sum_{j=n-m+1}^n Q_jD_{\ast j}|_C(x_{rs})
\end{eqnarray*}
Let $a_i$ be the constant term in $P_i$ and $b_j$ be the constant term in $Q_j$ so that 
\begin{eqnarray*}
\lefteqn{D|_C(x_{rs}) -\sum_{i=1}^{m}  a_iD_{i\ast}|_C(x_{rs})+\sum_{j=n-m+1}^n b_jD_{\ast j}|_C(x_{rs})} \\
& = &zx_{rs}-x_{rs}z + \sum_{i=1}^{m} (P_i-a_i)D_{i\ast}|_C(x_{rs}) +\sum_{j=n-m+1}^n (Q_j-b_j)D_{\ast j}|_C(x_{rs}).
\end{eqnarray*}
 As $D(x_{rs})$ is homogeneous of degree one, the terms on the left hand side of the equation all have degree one, whereas the nonzero terms on the right hand side all have degree greater than one. It follows that both sides are zero and so 
\begin{eqnarray*}
D|_C(x_{rs}) &=& \sum_{i=1}^{m} a_iD_{i\ast}|_C(x_{rs})+\sum_{i=n-m+1}^n  b_jD_{\ast j}|_C(x_{rs})\\
\end{eqnarray*}
Set $$D':= D-\sum_{i=1}^{m} a_iD_{i\ast}+\sum_{i=n-m+1}^n  b_jD_{\ast j}.$$
 Then $D'|_C=0$. Note that $D'$ is  a homogeneous derivation as $D$, $D_{i\ast}$ and $D_{\ast j}$ are all homogeneous.  Lemma~\ref{lemma-db-in-b} shows that $D'(B)\subseteq B$.  

\end{proof}



\subsection{Derivatives column by column}

Let $\epsilon_1,\dots,\epsilon_m$ and $\epsilon'_1,\dots,\epsilon'_n$ be the standard bases for $\mz^m$ and $\mz^n$, respectively. The quantum matrix algebra $\oqmmn$ has a natural $\mz^m\times\mz^n$ bigrading defined by giving each $x_{rs}$ bidegree, also called {\em bicontent}, $(\epsilon_r,\epsilon'_s)$. Observe that any quantum minor $[U\mid V]$ is homogeneous of bidegree $(\chi_U,\chi'_V)$, where $\chi_S$ (respectively $\chi'_S$) stands for the characteristic function of a subset $S$ of $\{1,\dots, m\}$ (respectively, of $\{1,\dots, n\}$). 
 If the bidegree of a homogeneous element is $(\chi_r, \chi_c)$, we refer to $\chi_r$ as the {\em row content} of the element and $\chi_c$ as the {\em column content} of the element. 
  The main use of the notion of content will be that in any equation of the form $P=Q$, the terms of the same bicontent on each side of the equation must be equal. \\

We consider $\oqmmn$ with $m<n$ and assume that we have a homogeneous derivation $D$ that acts trivially on $C$ and that $D(B)\subseteq B$.\\

  \begin{lemma}\label{lemma-linear-combination-column-derivatives}
  Use Notation~\ref{appendix-notation}. Let $D$ be a derivation on $\oqmmn$ such that $D(x_{ik})$ is homogeneous of degree one for all $x_{ik}$ and that $D(C)=0$.Then $D$ is a linear combination of $D_{*k}$, for $k=1,\dots, n-m$. 
 \end{lemma}

\begin{proof}
Let $x_{ik}\in B$, so that $1\leq i\leq m$ and $1\leq k \leq n-m$. 
As $D(B)\subseteq B$, by Lemma~\ref{lemma-db-in-b},  we can write 
$$D(x_{ik}) = \sum_{1\leq r\leq m\atop 1\leq s \leq n-m}a^{(ik)}_{rs}x_{rs}.$$
 for some $a^{(ik)}_{rs}\in\k$. Our first aim is to show that $a^{(ik)}_{rs}=0$ whenever $r\neq i$ so that 
 $$D(x_{ik}) =\sum_{1\leq s  \leq n-m} a^{(ik)}_{is}x_{is}.$$

Set $l:=n-m+1$. Choose any $j<i$. Note that $x_{ik}x_{jl} =x_{jl}x_{ik}$. 
As $x_{jl}\in C$ we know that $D(x_{jl})=0$ so that when we apply the derivative $D$ to this equation, we obtain
$$ \sum_{1\leq r\leq m\atop 1\leq s<l}a^{(ik)}_{rs}x_{rs}x_{jl} = \sum_{1\leq r\leq m\atop 1\leq s<l}a^{(ik)}_{rs}x_{jl}x_{rs}$$ 
or,
$$\sum_{1\leq r\leq m\atop 1\leq s<l}a^{(ik)}_{rs}\left(x_{rs}x_{jl}-x_{jl}x_{rs}\right)=0.$$
Consider terms in this equation with bicontent $(2\epsilon_j; \epsilon_s+\epsilon_l)$ (which from now on we also denote by $(j,j;s,l)$ to ease notation) to see that 
$$\sum_{1\leq s<l}a^{(ik)}_{js}\left(x_{js}x_{jl}-x_{jl}x_{js}\right)=0.$$ 
As $s<l$ this gives $$\sum_{1\leq s<l}a^{(ik)}_{js}(q-1)x_{jl}x_{js}=0,$$ 
From which it follows that $a^{(ik)}_{js}=0$ for all $s$ when $j<i$. \\

Next, choose any $j>i$. 
As $D(x_{jk})\in B$, write 
$$D(x_{jk})=\sum_{1\leq r\leq m\atop 1\leq s<l}a^{(jk)}_{rs}x_{rs}.$$
Apply $D$ to the equation $x_{ik}x_{jl} - x_{jl}x_{ik}=\qhat x_{il}x_{jk}$ (where $\qhat := q-q^{-1}$), noting  that $D(x_{*l})=0$ as $x_{*l}\in C$,
to obtain
\begin{equation}\label{equation-nasty-case}
\sum_{1\leq r\leq m\atop 1\leq s<l} a^{(ik)}_{rs}x_{rs}x_{jl}-\sum_{1\leq r\leq m\atop 1\leq s<l} a^{(ik)}_{rs}x_{jl}x_{rs} = \qhat\left(\sum_{1\leq r\leq m\atop 1\leq s<l} a^{(jk)}_{rs}x_{il}x_{rs}\right)
\end{equation}
Look at terms with bicontent $(j,j;s,l)$ in Equation~(\ref{equation-nasty-case}). There are no such terms on the right hand side, as $i$ occurs as a row index and $i\neq j$, while on the left hand side we get such terms when $r=j$. Hence, $a^{(ik)}_{js}x_{js}x_{jl} -a^{(ik)}_{js}x_{jl}x_{js} = 0$, for each $s$. As $s<l$ this gives 
$(q-1)a^{(ik)}_{js}x_{js}x_{jl} =0$ and so $a^{(ik)}_{js}=0$, for all $s$  and each $j>i$.  As we already know that $a^{(ik)}_{js}=0$ for all $s$ when  $j<i$ this gives 
$$D(x_{ik})=\sum_{1\leq s<l} a^{(ik)}_{is}x_{is},$$
for $1\leq i\leq m$ as required.\\

Fix $k$, with $1\leq k<l$, and $j>1$. Apply the derivative $D$ to the equation $x_{1k}x_{jl}-x_{jl}x_{1k}=\qhat x_{1l}x_{jk}$, using the expressions we have just derived, to obtain
\begin{equation}\label{equation-nasty-case-again}
\sum_{1\leq s<l} a^{(1k)}_{1s}x_{1s}x_{jl} - \sum_{1\leq s<l} a^{(1k)}_{1s}x_{jl}x_{1s}
=
\qhat \sum_{1\leq s<l} a^{(jk)}_{js}x_{1l}x_{js}
\end{equation}
Look at terms with content $(1,j;s,l)$ in Equation~(\ref{equation-nasty-case-again}) for fixed $s$. 
On the left hand side we have $a^{(1k)}_{1s}x_{1s}x_{jl} - a^{(1k)}_{1s}x_{jl}x_{1s}$ and this is equal to $\qhat a^{(1k)}_{1s}x_{1l}x_{js}$. 
On the right hand side we have $\qhat  a^{(jk)}_{js}x_{1l}x_{js}$. It follows that $a^{(jk)}_{js}=a^{(1k)}_{1s}$ for each $s$ with $1\leq s<l$, so that 
$$D(x_{jk})=\sum_{1\leq s<l} a^{(1k)}_{1s}x_{js},$$ 
for each $j$ with $1\leq j\leq m$.

Hence, 
\begin{align}\label{derivative-of-column}
D\begin{pmatrix} x_{1k}\\x_{2k}\\\vdots\\x_{mk}\end{pmatrix}
=a^{(1k)}_{11}\begin{pmatrix} x_{11}\\x_{21}\\\vdots\\x_{m1}\end{pmatrix}
+a^{(1k)}_{12}\begin{pmatrix} x_{12}\\x_{22}\\\vdots\\x_{m2}\end{pmatrix}
+\dots
+a^{(1k)}_{1,l-1}\begin{pmatrix} x_{1,l-1}\\x_{2,l-1}\\\vdots\\x_{m,l-1}\end{pmatrix}.
\end{align}

For each $k=1,\dots, l-1$, we want to show that $a^{(1k)}_{1r}=0$ when $r\neq k$, so that 
\begin{align*}
D\begin{pmatrix} x_{1k}\\x_{2k}\\\vdots\\x_{mk}\end{pmatrix}
=a^{(1k)}_{1k}\begin{pmatrix} x_{1k}\\x_{2k}\\\vdots\\x_{mk}\end{pmatrix}
=a^{(1k)}_{1k}D_{*k}\begin{pmatrix} x_{1k}\\x_{2k}\\\vdots\\x_{mk}\end{pmatrix}.
\end{align*}
~\\

Fix $k$ with $1\leq k\leq l-1$. 
First, we show that $a^{(1k)}_{1r}=0$ for each $r$ such that $k<r\leq l-1$. If $k=l-1$, there is no such $r$ to consider, so assume that $k<l-1$, in which case $k+1<l$ and $x_{1,k+1}\in B$. 
Apply $D$ to the equation 
$x_{1k}x_{1,k+1} =qx_{1,k+1}x_{1k}$ for $k+1<l$ to obtain 
\[
D(x_{1k})x_{1,k+1} +x_{1k}D(x_{1,k+1})= qD(x_{1,k+1})x_{1k}+qx_{1,k+1}D(x_{1k}),
\]
which gives
\begin{eqnarray}\label{adjacent-columns}
\sum_{1\leq s<l} a^{(1k)}_{1s}x_{1s}x_{1,k+1}+\sum_{1\leq s<l} a^{(1,k+1)}_{1s}x_{1k}x_{1s}
=\sum_{1\leq s<l} qa^{(1,k+1)}_{1s}x_{1s}x_{1k}+\sum_{1\leq s<l} qa^{(1k)}_{1s}x_{1,k+1}x_{1s}.
\end{eqnarray}
Consider terms in this equation with content $(1,1;k+1,r)$ for $r>k$. These occur in the first and fourth sums when $s=r$ and do not occur in the second and third sums because $k$ is in the column content of the terms in these sums. Hence,  
$a^{(1k)}_{1r}x_{1r}x_{1,k+1}=qa^{(1k)}_{1r}x_{1,k+1}x_{1r}$. If $r=k+1$, then we get $a^{(1k)}_{1,k+1}x_{1,k+1}x_{1,k+1}=qa^{(1k)}_{1,k+1}x_{1,k+1}x_{1,k+1}$ and so $a^{(1k)}_{1,k+1}=0$. If $r>k+1$ then $x_{1r}x_{1,k+1} =q^{-1}x_{1,k+1} x_{1r}$ so we see that 
$q^{-1}a^{(1k)}_{1r}x_{1,k+1}x_{1r}=qa^{(1k)}_{1r}x_{1,k+1}x_{1r}$ and so $a^{(1k)}_{1r}=0$. Hence, 
$a^{(1,k)}_{1,r}=0$ for all $r>k$.\\

Next, we show $a^{(1k)}_{1r}=0$ when $r<k$. If $k=1$ then there is no such $r$ to consider, so assume that $k>1$, in which case $k-1\geq 1$ so that $x_{1,k-1}$ exists and is in $B$. 
Apply $D$ to the the equation 
$x_{1,k-1}x_{1,k} =qx_{1,k}x_{1,k-1}$
\[
D(x_{1,k-1})x_{1,k} +x_{1,k-1}D(x_{1,k})= qD(x_{1,k})x_{1,k-1}+qx_{1,k}D(x_{1,k-1}),
\]
which gives
\begin{eqnarray}\label{adjacent-columns}
\sum_{1\leq s<l} a^{(1,k-1)}_{1s}x_{1s}x_{1,k}+\sum_{1\leq s<l} a^{(1,k)}_{1s}x_{1,k-1}x_{1s}
=\sum_{1\leq s<l} qa^{(1k)}_{1s}x_{1s}x_{1,k-1}+\sum_{1\leq s<l} qa^{(1,k-1)}_{1s}x_{1k}x_{1s}.
\end{eqnarray}
Consider terms in this equation with content $(1,1;r,k-1)$ with $r<k$. There are no such terms in the first and fourth terms, as $k$ is in the column content of the terms in these sums, and such terms occur in the second and third terms when $s=r$. Hence, 
 $a^{(1,k)}_{1r}x_{1,k-1}x_{1r}=qa^{(1,k)}_{1r}x_{1r}x_{1,k-1}$. When $r=k-1$ we see that $a^{(1,k)}_{1,k-1}x_{1,k-1}x_{1,k-1}=qa^{(1,k)}_{1,k-1}x_{1,k-1}x_{1,k-1}$ so that $a^{(1,k)}_{1,k-1}=0$. When $r<k-1$ then $x_{1,k-1}x_{1r}=q^{-1}x_{1r}x_{1,k-1}$ so we see that 
 $q^{-1}a^{(1,k)}_{1,r}x_{1r}x_{1,k-1}=qa^{(1,k)}_{1,r}x_{1,r}x_{1,k-1}$ and so $a^{(1,k)}_{1,r}=0$. Hence, $a^{(1,k)}_{1,r}=0$ for all $r<k$. Thus, $a^{(1k)}_{1r}=0$ whenever $r\neq k$ so that $D(x_{ik})=a^{(1k)}_{ik}x_{ik}$ for each $1\leq i\leq m$, as required to show that \[
D\begin{pmatrix} x_{1k}\\x_{2k}\\\vdots\\x_{mk}\end{pmatrix}
=a^{(1k)}_{1k}\begin{pmatrix} x_{1k}\\x_{2k}\\\vdots\\x_{mk}\end{pmatrix}
\]
for each $1\leq k<l$. As we already know that $D$ acts trivially on columns $n+1-m$ up to $n$, this gives 
$$D=  a^{(1,1)}_{11}D_{*1} +a^{(1,2)}_{12} D_{*2} +\dots+a^{(1,l-1)}_{1,l-1} D_{*,l-1},$$
as required.
\end{proof}


\subsection{Conclusion of appendix} 

The preceding analysis proves the following Proposition.

\begin{proposition}\label{proposition-non-square} 
 
 Suppose that $m<n$ and that we have a derivation $D$ of $\oqmmn$ such that

(i) the derivation $D$ is homogeneous; that is, all terms appearing non-trivially in $D(x_{ij})$ have degree one. 

(ii) $D([1,\dots,m\mid n+1-m,\dots,n])=0$.\\

Then 
\[
D=
\sum_{i=1}^{m} a_{i}D_{i*} + \sum_{k=1}^{n}\, b_{k}D_{*k}
\]
for some $a_i,b_k\in K$.
\end{proposition}

\begin{proof}
By Lemma~\ref{lemma-linear-combination}, there is a homogeneous derivation $D^{(1)}$ such that $D^{(1)}-D$ is a linear combination of the $D_{i*}$, with $1\leq i\leq m$, and $D_{*j}$, with $n-m+1\leq j\leq n$,  and such that $D^{(1)}|_C=0$. By Lemma~\ref{lemma-linear-combination-column-derivatives}, 
 $D^{(1)}$ is a linear combination of the $D_{*k}$ with $1\leq k\leq n-m$. Writing $D=D^{(1)}-\left(D^{(1)}-D\right)$ gives the required result. 
\end{proof}

\end{section} 






\ \\

\begin{minipage}{\textwidth}
\noindent S Launois\\
School of Mathematics, Statistics and Actuarial Science,\\
University of Kent\\
Canterbury, Kent, CT2 7FS,\\ UK\\[0.5ex]
email: S.Launois@kent.ac.uk \\

\noindent T H Lenagan\\
School of Mathematics and Maxwell Institute for Mathematical Sciences,\\
The University of Edinburgh,\\
Edinburgh, EH9 3FD,\\
UK\\[0.5ex]
email: t.lenagan@ed.ac.uk\\
\end{minipage}


\end{document}